\documentclass[reqno,11pt]{amsart}

\usepackage{amsmath,amssymb,amsthm,mathrsfs}
\usepackage{epsfig,color}
\usepackage[latin1]{inputenc}

\usepackage[numbers]{natbib}

\numberwithin{equation}{section}

\def\H{\mathcal H}

\def\R{\mathbb R}

\newcommand{\zero}{\mathbf{0}}

\newcommand{\dist}{\mathop{\mathrm{dist}}}
\def\e{\varepsilon}

\def\S{{\partial\Omega}}
\def\vphi{\varphi}
\def\Div{{\rm div}\,}
\def\om{\omega}
\def\l{\lambda}
\def\g{\gamma}

\def\Om{\Omega}
\def\de{\delta}
\def\Id{{\rm Id}}

\def\pa{\partial}

\def\00{{\bf 0}}

\def\SS{\mathbb S}

\def\E{\mathcal{E}}

\renewcommand{\a}{\alpha}
\renewcommand{\b}{\beta}

\newcommand{\D}{\Delta}
\renewcommand{\l}{\lambda}
\renewcommand{\L}{\Lambda}

\newcommand{\n}{\nabla}
\newcommand{\var}{\varphi}

\renewcommand{\om}{\omega}

\renewcommand{\Div}{{\rm div \,}}

\newcommand{\diam}{\mathrm{diam}}

\def\bb{\mathbf{b}}

\def\C{\mathbf{C}}
\def\D{\mathbf{D}}
\def\B{\mathcal{B}}

\newtheorem*{theorem*}{Theorem}
\newtheorem{theorem}{Theorem}[section]
\newtheorem{lemma}[theorem]{Lemma}
\newtheorem{proposition}[theorem]{Proposition}

\newtheorem{remark}[theorem]{Remark}

\title[Sharp stability for the nonlocal Alexandrov theorem]{Rigidity and sharp stability estimates  
for \\
hypersurfaces with constant and almost-constant \\ nonlocal mean curvature}
\author{G. Ciraolo}
\address{ Dipartimento di Matematica e Applicazioni,
Università di Palermo, Via Archirafi 34, 90123 Palermo, Italy}
\email{giulio.ciraolo@unipa.it}

\author{A. Figalli}
\address{Mathematics Department, The University of Texas at Austin,
2515 Speedway Stop C1200, RLM 8.100 Austin, TX 78712, USA}
\email{figalli@math.utexas.edu}

\author{F. Maggi}
\address{Mathematics Department, The University of Texas at Austin,
2515 Speedway Stop C1200, RLM 8.100 Austin, TX 78712, USA}
\email{maggi@math.utexas.edu}

\author{M. Novaga}
\address{Dipartimento di Matematica,
Universit\`a di Pisa,
Largo Bruno Pontecorvo 5,
56127 Pisa, Italy}
\email{novaga@dm.unipi.it}

\begin{document}

\begin{abstract}
We prove that the boundary of a (not necessarily connected) bounded smooth set with constant nonlocal mean curvature is a sphere. More generally, and in contrast with what happens in the classical case, 
we show that the Lipschitz constant of the nonlocal mean curvature of such a boundary controls 
its $C^2$-distance from a single sphere. The corresponding stability inequality is obtained with a sharp decay rate.
\end{abstract}

\maketitle

\section{Introduction} \label{sec:intro} The aim of this paper is to address a classical question in Differential Geometry, namely the characterization of compact embedded constant mean curvature surfaces as spheres -- Alexandrov's theorem \cite{A} -- in the case of surfaces with constant {\it nonlocal} mean curvature. The seminal papers \cite{CRS,CS} have drawn an increasing attention
to the geometry of nonlocal minimal surfaces, i.e., boundaries of sets $\Om\subset\R^n$ which are stationary for the $s$-perimeter functional
\[
P_s(\Om)=\int_\Om\int_{\Om^c}\frac{dx\,dy}{|x-y|^{n+2\,s}}\,,\qquad \Om^c = \R^n \setminus \Om\,,
\]
corresponding to some value of $s\in(0,1/2)$ (see for instance \cite{CV,ADM,SV,BFV,FV,DDW,F2M3,DDDV}). If $\Om$ is an open set with smooth boundary and $A\subset\R^n$ is an open set, then the condition
\[
\de P_s(\Om)[X]=\frac{d}{dt}\,P_s(\Phi_t(\Om))=0\,,\qquad\forall \,X\in C^\infty_c(A;\R^n)\,,
\]
(where $\Phi_t$ denotes the flux defined by the vector-field $X$) is equivalent to require the vanishing of the nonlocal mean curvature $H_s^\Om (p)$ of $\Om$ at every point $p\in A\cap \pa \Om$.
More in general, we say that $H_s^\Om:\pa \Om\cap A \to \R$ is the nonlocal mean curvature of $\pa \Om$
inside $A$ if
$$
\frac{d}{dt}\,P_s(\Phi_t(\Om))=\int_{\pa \Om}H_s^\Om(x)\,X(x)\cdot \nu_x\, d\H^{n-1}_x \qquad\forall \,X\in C^\infty_c(A;\R^n)\,.
$$
Here $\nu_x$ is the exterior unit normal to $\Om$ at $x\in \pa\Om$,
and $\H^{n-1}$ denotes the $(n-1)$-Hausdorff measure.

Whenever $\pa \Om$ is sufficiently smooth (say $\pa\Om \in C^{1,\a}$ for some $\a>2s$),
one can show that the nonlocal mean curvature of $\pa \Om$ at a point $p \in \pa \Om$ is given by
\begin{equation} \label{def:Hs}
H_s^\Om (p) = \frac{1}{\om_{n-2}}
\int_{\R^n} \frac{\widetilde\chi_\Om(x)}{|x-p|^{n+2s}} \,dx \,,\qquad \widetilde\chi_\Om(x)=\chi_{\Om^c}(x)-\chi_\Om(x)\,,
\end{equation}
where
$\chi_E$ denotes the characteristic function of a set $E$, $\om_{n-2}$ is the measure of the $(n-2)$-dimensional sphere, and the integral is defined in the principal value sense (see for instance \cite[Theorem 6.1 and Proposition 6.3]{F2M3}).
It is useful to keep in mind that, by means of the divergence theorem, the nonlocal mean curvature can also be computed as a boundary integral, that is
 \begin{equation} \label{def:Hsdiv}
 H_s^\Om (p) = \frac{1}{s\, \om_{n-2}} \int_{\pa \Om} \frac{(x-p)\cdot \nu_x}{|x-p|^{n+2s}}\, d\H^{n-1}_x \,.
 \end{equation}

Our main interest here is describing the shape of open sets $\Om$ having constant, or almost-constant, nonlocal mean curvature. In this direction we obtain three main results. 

The first one is a nonlocal version of the classical Alexandrov's 
theorem \cite{A}:

\begin{theorem}\label{thm rigidity}
If $\Om$ is a bounded open set of class $C^{1,2s}$ and $H_s^\Om$ is constant on $\pa\Om$, then $\pa\Om$ is a sphere.
\end{theorem}

In our second result we prove that if $H_s^\Om$, instead of being constant, has just a small Lipschitz constant
\begin{equation}
\label{def:Alexandrov_deficit}
\de_s(\Om) = \sup_{p,q \in \S,\, p\ne q}  \frac{|H_s^\Om (p) - H_s^\Om (q)|}{|p-q|}\,,
\end{equation}
and if $\pa\Om$ is of class $C^{2,\a}$ for some $\a>2s$, then $\pa\Om$ is close to a sphere,
with a {\em sharp} estimate in terms of $\delta_s(\Om)$. To state our result we introduce the following uniform distance from being a ball:
\[
\rho(\Om)=\inf\Big\{\frac{t-s}{\diam(\Om)}:p\in\Om\,,B_s(p)\subset\Omega\subset B_t(p)\Big\}\,.
\]

\begin{theorem} \label{thm:stability}
If $\Om$ is a bounded open set with $C^{2,\a}$-boundary for some $\a>2s$, then
there exists a dimensional constant $\hat C(n)$ such that
\begin{equation}
  \label{stima stab}
\rho(\Om)\le \hat C(n)\,\eta_s(\Om)\,,
\end{equation}
where
\begin{equation}
\label{def:eta}
\eta_s(\Om)=\frac{\diam(\Om)^{2n+2s+1}}{|\Om|^2}\,\de_s(\Om)\,.
\end{equation}
Moreover there exists $\eta(n)>0$ such that if $\eta_s(\Om)\le\eta(n)$ then, up to rescaling $\Om$, we can find a bi-Lipschitz map $F:\pa B_1(\zero)\to\pa\Om$ satisfying
\begin{equation}
  \label{lipschitz map}
\Bigl(1-\bar C(n)\,\sqrt{\eta_s(\Om)}\Bigr)|x-y|\le |F(x)-F(y)|\le \Bigl(1+\bar C(n)\,\sqrt{\eta_s(\Om)}\Bigr)|x-y|\qquad\forall \,x,y\in\pa B_1(\zero)
\end{equation}
for some dimensional constant $\bar C(n)>0$.
\end{theorem}

\begin{remark}
  {\rm Note that both $\rho(\Om)$ and $\eta_s(\Om)$ are scaling invariant quantities.
  Also, the estimate \eqref{stima stab} is {\em optimal} in terms of the exponent of $\eta_s(\Om)$, as it can be easily seen by considering
  a sequence of ellipsoids converging to
the unit ball.}
\end{remark}

\begin{remark}
\label{rmk:stability local}
  {\rm If $\Om$ is an open set with $C^2$-boundary then $(1-2\,s)\,H_s^\Om\to H^\Om$ on $\pa\Om$ as $s\to(1/2)^-$, where $H^\Om$ is the classical mean curvature of $\pa\Om$ (see \cite{AV}). Therefore,
  because of the scaling factor $(1-2s)$ one cannot obtain any information from Theorem \ref{thm:stability} in the limit $s\to(1/2)^-$. This is not a drawback of our result, as its local analog is false. Indeed, one can construct examples of connected boundaries whose classical mean curvature is arbitrarily close to a constant in $C^1$ topology, but these sets are close (in the Hausdorff distance) to a union of tangent spheres of equal radii \cite{B}. }
\end{remark}

Both results above are obtained by the moving planes method. Note that the use of this method in obtaining stability estimates is well-established in the local case, see for example \cite{ABR,CMS2,CMV2} in relation to Serrin's overdetermined problem and \cite{CV} concerning Alexandrov's theorem.
Also, this method has already been successfully used in some nonlocal settings
to obtain symmetry results 
 (see for instance \cite{Rei,BMS} and the references therein).\\

Once Theorem \ref{thm:stability} is proved, we can exploit the regularity theory for nonlocal equations in order to obtain a sharp stability estimate in stronger norms.
Indeed, by a careful analysis we can conclude that $\pa\Om$ is close in $C^2$ to a sphere with a {\em linear} control in terms of $\eta_s(\Om)$, exactly as in \eqref{stima stab}. In particular the following result improves
the estimate in \eqref{lipschitz map}, although its proof relies on more delicate tools
(and actually \eqref{lipschitz map} is needed in the proof of this result).

\begin{theorem} \label{thm:stability2}
Assume that $\Om$ is a bounded open set with $C^{2,\a}$-boundary for some $\a>2s$,
and suppose that $\Om$ has been translated and rescaled so that
\begin{equation}
\label{eq:scale omega}
B_{1-2\rho(\Om)}(\zero)\subset \Om\subset B_{1}(\zero).
\end{equation}
There exists $\eta(n,s)>0$ such that the following holds: If $\eta_s(\Om)\le\eta(n,s)$ 
then  there is a map
$F: \pa B_1(\zero)\to \R^n$ of class $C^{2,\tau}$ for any $\tau<2s$, such that $F(\pa B_1(\zero))=\pa \Om$ and
$$
\|F-\Id\|_{C^{2,\tau}(\pa B_1(\zero))}\le C(n,s,\tau)\,\eta_s(\Om).
$$
In particular, if $\eta_s(\Om)$ is sufficiently small then $\Om$ is a convex domain.
\end{theorem}

\medskip

We conclude this introduction by emphasizing  that
boundaries with constant or almost constant mean curvature behave differently in the nonlocal and in the local case, the former setting being much more rigid than the latter. Indeed, as proven in Theorem \ref{thm rigidity},
even without any connectedness assumption
a boundary with constant nonlocal mean curvature is a {\it single} sphere, whereas  of course any disjoint union of spheres with equal radii has constant mean curvature in the classical sense.
Actually, even working only with connected boundaries, a significant difference arises at the level of stability.
Indeed, as already mentioned in Remark \ref{rmk:stability local}, a connected boundary with almost-constant mean curvature may be close to a compound of nearby spheres of equal radii 
(unless one imposes some strong geometric constraints on the considered set, like a uniform ball condition \cite{CV} or an upper volume density bound \cite{CM}).
In contrast with this picture, as shown in Theorems \ref{thm:stability} and \ref{thm:stability2} above, uniformly bounded sets with almost-constant nonlocal mean curvature must be close to a single ball without the need to any uniform control in their geometry. This points out an interesting feature of the nonlocal case, namely, the nonlocality of the underlying perimeter functional prevents bubbling phenomena (in the limit $\de_s(\Om)\to 0$).

We also note that, as it will be apparent from our arguments, Theorems \ref{thm rigidity} and \ref{thm:stability} hold (with different constants and possibly without scale invariant statements) if in the definition of $H_s^\Om$ one replaces the kernel $|z|^{-n-2s}$ with $k(|z|)$, where
\[
k(t)>0\,,\qquad t^{n+2s}k(t)+t^{n+2s+1}|k'(t)|\le C\,,\qquad \sup_{(0,t)}k'\le c(t)<0\,,\qquad\forall\, t>0\,.
\]
For the validity of Theorem \ref{thm:stability2}, one needs to impose the additional constraint that $k(t)$ behaves as a smooth perturbation of $t^{-(n+2s)}$ as $t\to 0^+$.
\\

This paper is organized as follows. In section \ref{sec:prelim} we prove a technical fact about approximating the nonlocal mean curvature in $C^1$
with nonlocal ``curvatures'' coming from smooth kernels. Then in section \ref{sec:symmetry} we prove the nonlocal version of Alexandrov's theorem, while in section \ref{sec:stab} we address the stability analysis.
\\

After the writing of this paper was completed we learned that, at the very same time
and independently of us,
X. Cabr\'e, M. Fall, J. Sola-Morales,
and
T. Weth have proved a result analogous to our Theorem \ref{thm rigidity} \cite{CFSW}.\\


\noindent {\bf Acknowledgment:}
This work has been done while GC and MN
where visiting the University of Texas at Austin, under the support of NSF-DMS FRG Grant 1361122.
In addition, GC is supported by a Oden Fellowship at ICES, the GNAMPA of the Istituto Nazionale di Alta Matematica (INdAM) and the FIRB project 2013 ``Geometrical and Qualitative aspects of PDE''. AF is supported by NSF Grant DMS-1262411, and
FM is supported by NSF-DMS Grant 1265910.

\section{A technical lemma} \label{sec:prelim}
In order to perform our computations, and in particular to avoid integrability issues, it will be useful to work with smooth kernels. We thus consider the approximation $K_\e(x)=\vphi_\e(|x|)$ of the kernel $K(x)=\frac1{\om_{n-2}}|x|^{-(n+2s)}$ corresponding to a choice of $\vphi_\e\in C_c^\infty([0,\infty))$ such that $\vphi_\e\ge0$, $\vphi_\e'\le0$, and
\begin{equation}
\label{eq:vare2}
\left\{\begin{split}
&t^{n+2s}\,\var_\e(t)+t^{n+2s+1}\,|\var_\e'(t)|\le C(n,s)\,,
\\
&|\var_\e'(t)|\uparrow \frac{n+2s}{\om_{n-2}}\frac{1}{t^{n+2s+1}}\qquad\text{as $\e\to 0^+$},
\end{split}\right .
\qquad\forall\,t>0\,.
\end{equation}
Note that this implies that, as $\e\to 0$,
\begin{equation}
  \label{dove va vphi}
  \vphi_\e(t)\uparrow\frac{1}{\om_{n-2}\,t^{n+2s}}\qquad\forall \,t>0\, ,
\end{equation}
and both $\vphi_\e$ and $\vphi_\e'$ converge to their limits uniformly on $[t_0,\infty)$ for every fixed $t_0>0$.

Let us define 
\begin{equation}
  \label{weset}
  H_{s,\e}^\Om (p) =
\int_{\R^n} \widetilde\chi_\Om(x)\,\var_\e\bigl(|x-p|\bigr)\, dx\,,\qquad  p \in \pa\Om\,.
\end{equation}
Then, arguing as in \cite[Proposition 6.3]{F2M3} we find that
\begin{equation}
  \label{uniform conv}
  \lim_{\e\to 0}\|H_{s,\e}^\Om - H_{s}^\Om\|_{C^0(\pa \Om)}=0\,,
\end{equation}
provided $\Om$ is a bounded open set with $C^{1,\a}$-boundary for some $\a>2s$. We now prove the following technical fact.

\begin{lemma} \label{lemma:derivata_Hs} Assume that $\Om$ is a bounded open set with $C^{2,\a}$-boundary for some $\a>2s$. Then $H_{s}^\Om\in C^1(\pa\Om)$ and $H_{s,\e}^\Om\to H_s^\Om$ in $C^1(\pa\Om)$ as $\e\to 0$.
\end{lemma}

\begin{proof} 
Since we already know that $H_{s,\e}^\Om$ converge to $H_{s}^\Om$ in $C^0$ (see \eqref{uniform conv}), it is enough to 
prove that $H_{s,\e}^\Om$ is a Cauchy sequence in $C^1$, that is
\begin{equation}
  \label{that}
  \lim_{(\e,\eta)\to (0,0)}\|\n H_{s,\e}^\Om-\n H_{s,\eta}^\Om\|_{C^0(\pa\Om)}=0\,.
\end{equation}
To this end we first notice that, by setting
$$
\psi_\e(t)=-\frac{1}{t^{n}}\int_t^\infty \var_\e(\tau)\,\tau^{n-1}\,d\tau\qquad \forall\,t>0
$$
we have
$$
\Div \bigl(x\,\psi_\e(|x|)\bigr)=n\,\psi_\e(|x|)+|x|\,\psi_\e'(|x|)=\var_\e(|x|) \qquad \forall\,x \in \R^n\,,
$$
hence $H_{s,\e}^\Om$ can be rewritten as
\begin{equation}
  \label{ciaociao}
  H_{s,\e}^\Om(p)=-2\int_{\pa\Om}\psi_\e\bigl(|x-p|\bigr)\,(x-p)\cdot \nu_x\,d\H^{n-1}_x\qquad\forall \,p\in\pa\Om\,.
\end{equation}
Note that $\psi_\e$ is smooth, it satisfies
\begin{equation}
  \label{stime psi}t^{n+2s}\,\psi_\e(t)+t^{n+2s+1}\,|\psi_\e'(t)|+t^{n+2s+2}\,|\psi_\e''(t)|\le C(n,s)  \qquad\forall \,t>0
\end{equation}
(thanks to \eqref{eq:vare2}), and both $\psi_\e$ and $\psi_\e'$ converge uniformly to their limits on $[t_0,\infty)$ for every fixed $t_0>0$ as $\e\to 0$.

Now, given $p\in\pa\Om$ and $\hat e\in T_p(\pa\Om)\cap \mathbb S^{n-1}$ a tangent vector, by the smoothness of $\psi_\e$ one finds
\begin{equation}
\label{eq:DHe}
\n H_{s,\e}^\Om(p)\cdot \hat e=
2\int_{\pa\Om}\Bigl(\psi_\e\bigl(|x-p|\bigr)\,\nu_x\cdot \hat e+\frac{\psi_\e'\bigl(|x-p|\bigr)}{|x-p|}\,[(x-p)\cdot \nu_x]\,[(x-p)\cdot \hat e]\Bigr)\,\,d\H^{n-1}_x\,.
\end{equation}
Up to decomposing $\R^n=\R^{n-1}\times\R$ so that $x=(\hat{x},x_n)$ denotes the generic point in $\R^n$, and up to translating $p$ into the origin $\zero$, 
we define 
\begin{equation}
\label{eq:DC}
\D_\rho=\{\hat{x}\in\R^{n-1}:|\hat{x}|<\rho\}\,,\qquad\C_\rho=\D_\rho\times(-\rho,\rho)\,,
\end{equation}
and we see that the smoothness of $\pa \Om$ implies that, up to a rotation, there exist $\rho>0$ and a function $f\in C^{2,\a}(\D_\rho)$, with $f(0)=\nabla f(0)=0$ and $\|f\|_{C^{2,\a}(\D_\rho)}\le L$,
such that
\[
\C_\rho\cap\pa\Om=\big\{(\hat{x},f(\hat{x})):\hat{x}\in\D_\rho\big\}=(\Id\times f)(\D_\rho)
\]
(by compactness of $\pa\Om$, both $\rho$ and $L$ are independent of the point $p \in \pa\Om$ under consideration). 
Now, if we set $\B_r=(\Id\times f)(\D_r)$ for $r \in (0,\rho)$, then by the uniform convergence of $\psi_\e$ and $\psi_\e'$ on $[r,\infty)$ we find that
\begin{multline}
\label{eq:HC1}
\bigg|\int_{\pa\Om\setminus \mathcal B_r}\Bigl(\psi_\e\bigl(|x|\bigr)\,\nu_x\cdot \hat e+\frac{\psi_\e'\bigl(|x|\bigr)}{|x|}\,(x\cdot \nu_x)\,(x\cdot \hat e)\Bigr)\,\,d\H^{n-1}_x\\
-\int_{\pa\Om\setminus \mathcal B_r}\Bigl(\psi_{\eta}\bigl(|x|\bigr)\,\nu_x\cdot \hat e+\frac{\psi_\eta'\bigl(|x|\bigr)}{|x|}\,(x\cdot \nu_x)\,(x\cdot \hat e)\Bigr)\,\,d\H^{n-1}_x\bigg|\to 0
\end{multline}
as $\e,\eta \to 0$. 

On the other hand, having in mind \eqref{eq:DHe} and \eqref{eq:HC1} and noticing that $\nu_x=\frac{(-\nabla f(\hat x),1)}{\sqrt{1+|\nabla f(\hat x)|^2}}$
for $x=(\hat x,f(\hat x)) \in \B_r$,
taking into account that $\hat{e}\cdot e_n=0$ for every $\hat{e}\in T_\zero(\pa\Om)$ we see that
\begin{align}\nonumber
&2\int_{\mathcal B_r}\Bigl(\psi_\e\bigl(|x|\bigr)\,\nu_x\cdot \hat e+\frac{\psi_\e'\bigl(|x|\bigr)}{|x|}\,(x\cdot \nu_x)\,(x\cdot \hat e)\Bigr)\,\,d\H^{n-1}_x
\\
&=\label{ciaociao2}
  2\int_{\D_r}
\biggl(-\psi_\e\Bigl(\sqrt{|\hat{x}|^2+f^2}\Bigr)\,\n f\cdot \hat e+\frac{\psi_\e'\left(\sqrt{|\hat{x}|^2+f^2}\right)}{\sqrt{|\hat{x}|^2+f^2}}\,[f- \n f\cdot \hat{x}]\,(\hat{x}\cdot \hat e)\biggr)\,\,d\hat{x}\,,
\end{align}
where $f=f(\hat{x})$ and $\nabla f=\nabla f(\hat{x})$. To get a good control on the above quantity, we symmetrize it with respect to $\hat{x}$ by performing the change of variable $\hat{x}\mapsto -\hat{x}$ and then add the two expressions (the one with the variable $\hat{x}$
and the one with $-\hat{x}$). In this way we see that the integral in \eqref{ciaociao2} is equal to
\begin{align*}
&
-\int_{\D_r}
\psi_\e\Bigl(\sqrt{|\hat{x}|^2+f(\hat{x})^2}\Bigr)\,\Bigl(\n f(\hat{x})\cdot\hat  e+\n f(-\hat{x})\cdot \hat e\Bigr)\,d\hat{x}\\
&+\int_{\D_r}\biggl(\psi_\e\Bigl(\sqrt{|\hat{x}|^2+f(\hat{x})^2}\Bigr)-\psi_\e\Bigl(\sqrt{|\hat{x}|^2+f(-\hat{x})^2}\Bigr)\biggr)\,\n f(-\hat{x})\cdot \hat e\,d\hat{x}\\
&+\int_{\D_r}
\frac{\psi_\e'\left(\sqrt{|\hat{x}|^2+f(\hat{x})^2}\right)}{\sqrt{|\hat{x}|^2+f(\hat{x})^2}}\,\Bigl([f(\hat{x})-f(-\hat{x})]\,(\hat{x}\cdot \hat e)
-\,[ \n f(\hat{x})\cdot \hat{x}+ \n f(-\hat{x})\cdot \hat{x}]\,(\hat{x}\cdot \hat e)\Bigr)\,d\hat{x}\\
&+\int_{\D_r}
\biggl(\frac{\psi_\e'\left(\sqrt{|\hat{x}|^2+f(\hat{x})^2}\right)}{\sqrt{|\hat{x}|^2+f(\hat{x})^2}}-\frac{\psi_\e'\left(\sqrt{|\hat{x}|^2+f(-\hat{x})^2}\right)}{\sqrt{|\hat{x}|^2+f(-\hat{x})^2}}\biggr)\,[f(-\hat{x})+ \n f(-\hat{x})\cdot \hat{x}]\,(\hat{x}\cdot \hat e)\,d\hat{x}.
\end{align*}
Hence, since $f(0)=\n f(0)=0$ and recalling \eqref{stime psi}, we can find a constant $C$, depending only on $n,s,L$, such that, for $|\hat{x}|<\rho$,
$$
|\n f(\hat{x})\cdot \hat e+\n f(-\hat{x})\cdot \hat e| \leq C\,|\hat{x}|^{1+\a},\qquad |\n f(-\hat{x})|\leq C|\hat{x}|,\qquad |f(\hat{x})|\leq C|\hat{x}|^2 \,,
$$
$$
|f(\hat{x})-f(-\hat{x})|\leq C\,|\hat{x}|^{2+\a},\qquad |\n f(\hat{x})\cdot \hat{x}+ \n f(-\hat{x})\cdot \hat{x}| \leq C|\hat{x}|^{2+\a} \,,
$$
$$
\biggl|\psi_\e\Bigl(\sqrt{|\hat{x}|^2+f(\hat{x})^2}\Bigr)-\psi_\e\Bigl(\sqrt{|\hat{x}|^2+f(-\hat{x})^2}\Bigr)\biggr|
\leq
C\frac{\bigl|f(\hat{x})^2-f(-\hat{x})^2\bigr|}{|\hat{x}|^{n+2s+2}}
\leq \frac{C}{|\hat{x}|^{n+2s-2-\a}},
$$
$$
\biggl|\frac{\psi_\e'\left(\sqrt{|\hat{x}|^2+f(\hat{x})^2}\right)}{\sqrt{|\hat{x}|^2+f(\hat{x})^2}}-\frac{\psi_\e'\left(\sqrt{|\hat{x}|^2+f(-\hat{x})^2}\right)}{\sqrt{|\hat{x}|^2+f(-\hat{x})^2}}\biggr|\leq
C\frac{\bigl|f(\hat{x})^2-f(-\hat{x})^2\bigr|}{|\hat{x}|^{n+2s+4}}
\leq \frac{C}{|\hat{x}|^{n+2s-\a}},
$$
thus
\begin{equation}
\label{eq:HC2}
\biggl|
\int_{\mathcal B_r}\Bigl(\psi_\e\bigl(|x|\bigr)\,\nu_x\cdot \hat e+\frac{\psi_\e'\bigl(|x|\bigr)}{|x|}\,(x\cdot \nu_x)\,(x\cdot \hat e)\Bigr)\,\,d\H^{n-1}_x\biggr| \leq C\,r^{\a-2s},
\end{equation}
where $C$ depends only on $n$, $s$ and $L$. Therefore, combining \eqref{eq:HC1} and \eqref{eq:HC2} we obtain
$$
\limsup_{\e,\eta\to 0}\bigl|\n H_{s,\e}^\Om(\zero)\cdot \hat e-\n H_{s,\eta}^\Om(\zero)\cdot \hat e\bigr|\leq C\,r^{\a-2s}\,,
$$
for every $r\in(0,\rho)$ and any unit tangent vector $\hat e \in T_\zero(\pa \Om)\cap \SS^{n-1}$. Hence, by letting $r\to 0^+$ we conclude the proof.
\end{proof}

%

\section{Symmetry and the Nonlocal Alexandrov Theorem} \label{sec:symmetry} We start by introducing the notation used in exploiting the moving planes method. Given $e\in\SS^{n-1}$, $A\subset\R^n$, and $\mu\in\R$, we set
\begin{equation} \label{def:movingplanes}
\begin{array}{lll}
&\pi_{\mu}=\{ x\in\R^n: x\cdot e=\mu\} \ &\mbox{a hyperplane orthogonal to $e,$}\\
& \E_\mu=\{ x\in\R^n: x\cdot e>\mu\}\ &\mbox{the half-space on the ``positive'' side (with respect to $e$) of $\pi_\mu$,}\\
& A_{\mu}=\Om \cap \E_\mu &\mbox{the ``positive'' cap of $A$},\\
& x'_\mu=x-2\,(x\cdot e-\mu)\, e \ &\mbox{the reflection of $x$ with respect to $\pi_{\mu}$},\\
& A'_\mu=\{x'_\mu:x\in A\}  &\mbox{the reflection of $A$ with respect to $\pi_\mu$}.
\end{array}
\end{equation}
Now, if $\Omega$ is an open bounded (not necessarily connected) set in $\R^n$ with $C^1$-boundary and $\L =\sup\{x\cdot e: x\in \Om\}$, then for every $\mu<\L$  sufficiently close to $\L$ the reflection with respect to $\pi_\mu$ of the positive cap $\Om_\mu$ is contained in $\Om$, so it makes sense to define
\begin{equation}\label{def:lam_critico}
\l=\inf\bigl\{\mu\in\R: (\Om_{\tilde\mu})_{\tilde\mu}'\subset \Om \mbox{ for all } \tilde\mu\in(\mu,\L)\bigl\}\,.
\end{equation}
In the sequel, given a direction $e\in \pa B_1(\zero)$, $\pi_\l$ and $\Om_\l$ will be referred to as the {\it critical hyperplane} and the {\it critical cap} respectively, and for the sake of simplicity we will set
\[
x'=x'_\l=x-2\,(x\cdot e-\l)\, e\,,\qquad \Om'=\Om'_\l=\{x':x\in\Om\}\,.
\]
With this notation at hand, we recall from \cite{A} that for every direction $e$ at least one of the following two conditions always holds:

\medskip

\noindent {\it Case 1}: $\pa \Om_{\l}'$ is tangent to $\pa \Om$ at some
point $p' \in \pa \Om$, which is the reflection in $\pi_\l$ of a point $p\in \pa \Om_\l\setminus\pi_{\l}$;

\medskip

\noindent {\it Case 2}: $\pi_{\l}$ is orthogonal to $\pa \Om$ at some point $q\in\pa \Om\cap\pi_{\l}$.

\medskip

\noindent Both our main results will be based on the analysis of these two possibilities, under the assumption that $\de_s(\Om)=0$ or that $\de_s(\Om)$ is small, respectively.\\

We now prove the following result showing that $\delta_s(\Om)$ controls the $L^1$-distance between $\Om$ and $\Om'$
(recall that, given two sets $E$ and $F$, $E\triangle F$ denotes the symmetric difference of the two sets, that is $E\triangle F=(E\setminus F)\cup(F\setminus E)$).
Actually, to be able to obtain a sharp stability estimate in Theorem \ref{thm:stability},
it will be important to prove a stronger bound on $|\Om \triangle \Om'|$ when the set $\Om$ is already comparable to a ball of radius $1$ (see statement (b) below).

\begin{proposition} \label{prop:1direction}
Assume $\Om$ is a bounded open set with $C^{2,\a}$-boundary for some $\a>2s$,
fix $e \in \SS^{n-1}$,
and let $\Om'$ denote the reflection of $\Om$ with respect to the critical hyperplane $\pi_\l$.
\begin{enumerate}
\item[(a)] The bound
\begin{equation}
\label{diff_symm_small}
|\Om \triangle \Om'| \leq C_1\,\diam(\Om)^{n+s+(1/2)}\, \sqrt{\de_s(\Om)}
\end{equation}
holds with
\begin{equation} \label{def:C0}
C_1=2\,\sqrt{\frac{2\,\om_{n-2} }{n+2s}}\,.
\end{equation}
\item[(b)] Assume in addition that $\dist(\zero,\pi_\l)\leq 1/8$
and $B_r(\zero)\subset \Om\subset B_R(\zero)$ for some radii satisfying
\begin{equation}
\label{eq:radii}
\frac{1}{2} \leq r \leq R \leq 2,\qquad R-r \geq 16\,\delta_s(\Om).
\end{equation}
Then there exists a dimensional constant $C(n)$ such that
\begin{equation}
  \label{C1}
|\Om \triangle \Om'| \leq C(n)\,\sqrt{\de_s(\Om)}\sqrt{R-r}.
\end{equation}
\end{enumerate}
\end{proposition}

\begin{proof}
We first prove that
\begin{equation}
\label{bound_deficit_one_direct}
\int_{\Om \triangle \Om'} \dist(x,\pi_\l) \,dx \leq \frac{\om_{n-2} }{n+2s}\,\diam(\Om)^{n+2s+2}\, \de_s(\Om).
\end{equation}
Without loss of generality we let $e=e_1$. Let us first assume to be in case 1, that is, there exists $p\in \pa \Om_\l\setminus\pi_{\l}$ such that $p \in \pa \Om\cap\pa \Om'$. Then
\begin{equation}\label{H_diff_1}
\begin{split}
H_s^\Om(p) - H_s^{\Om}(p')&= H_s^\Om(p)-H_s^{\Om'}(p)\\ & = \frac{2}{\om_{n-2}} \biggl( \int_{\Om' \setminus \Om} \frac{1}{|x-p|^{n+2s}} \,dx - \int_{\Om \setminus \Om'} \frac{1}{|x-p|^{n+2s}} \,dx \biggr) \\
& = \frac{2}{\om_{n-2}}  \int_{\Om' \setminus \Om} \left( \frac{1}{|x-p|^{n+2s}}  - \frac{1}{|x'-p|^{n+2s}} \right) \,dx \, ,
\end{split}
\end{equation}
where all the integrals are intended in the principal value sense.
Since $x'=(2\l -x_1,x_2,...,x_n)$,
\begin{equation*}
\begin{split}
\frac{1}{|x-p|^{n+2s}}  - \frac{1}{|x'-p|^{n+2s}} & = \frac{1}{|x'-p|^{n+2s}} \Big[ \Big( \frac{|x'-p|}{|x-p|} \Big)^{n+2s} - 1 \Big] \\
& = \frac{1}{|x'-p|^{n+2s}} \Big[ \Big( 1 + \frac{4(x_1-\l)(p_1-\l)}{|x-p|^2} \Big)^{\frac{n+2s}{2}} - 1 \Big] \,,
\end{split}
\end{equation*}
by the convexity of the function $f(t)=(1+t)^{(n+2s)/2}-1$ we get that if $x\in\Om'$ then
\begin{equation} \label{Nucleo_diff_1}
\begin{split}
\frac{1}{|x-p|^{n+2s}}  - \frac{1}{|x'-p|^{n+2s}}  \geq \frac{2(n+2s) (x_1-\l)(p_1-\l)}{|x'-p|^{n+2s} |x-p|^2} \geq \frac{2(n+2s) (x_1-\l)(p_1-\l)}{\diam(\Om)^{n+2s+2} } \, ,
\end{split}
\end{equation}
where we used the fact that, by construction, $p'\in\pa\Om$ and therefore $|x-p|=|x'-p'|\le\diam(\Om)$ for every $x\in\Om'$. Since $x_1-\l \geq 0$ inside $\Omega'\setminus\Omega$
and $|p-p'|=2(p_1-\l)$, combining \eqref{H_diff_1} and \eqref{Nucleo_diff_1} we find
\begin{equation*}
\begin{split}
\de_s(\Om) &\geq \frac{H_s^{\Om}(p)-H_s^\Om(p')}{2(p_1-\l)} \geq \frac{2(n+2s)}{\diam(\Om)^{n+2s+2}\om_{n-2} } \int_{\Om' \setminus \Om} (x_1-\l)\, dx\\
&= \frac{(n+2s)}{\diam(\Om)^{n+2s+2}\om_{n-2} } \int_{\Om' \Delta \Om} |x_1-\l|\, dx\,,
\end{split}
\end{equation*}
which proves \eqref{bound_deficit_one_direct} in the first case. \\

We now assume that $\pi_{\l}$ is orthogonal to $\pa \Om$ at some point $q\in\pa \Om\cap\pi_{\l}$. Thanks to Lemma \ref{lemma:derivata_Hs} and \eqref{weset}, setting $u_\e(x)=\vphi_\e(|x-q|)$ we have
$$
\n H_s^\Om(q)\cdot e_1=\lim_{\e \to 0} \n H_{s,\e}^\Om(q)\cdot e_1
=-\lim_{\e \to 0} \int_{\R^n}\widetilde\chi_\Om(x)\,\nabla u_\e(x)\cdot e_1\,dx
=-2\lim_{\e \to 0} \int_{\Om} \nabla u_\e(x)\cdot e_1\,dx
$$
where we used that $\int_{\R^n}\nabla u_\e=0$. Since $\nabla u_\e(x)\cdot e_1=\vphi'_\e(|x-q|)\,\frac{(x-q)\cdot e_1}{|x-q|}$ is odd with respect to the hyperplane $\{x_1=\l\}$ (notice that $\l=q_1$) and $\l$ is the critical value for $e_1$, we find that $\int_{\Om\cap \Om'}\nabla u_\e\cdot e_1=0$, hence
$$
\n H_s^\Om(q)\cdot e_1=2\lim_{\e\to 0}\int_{\Om\setminus \Om'}\bigl|\var_\e'(|x-q|)\bigr|\,\frac{(x-q)\cdot e_1}{|x-q|}\,dx\,.
$$
We now observe that $\Om\setminus \Om'$ is contained inside the half-space $\{x_1\leq \l\}$ where the function
$\frac{(x-q)\cdot e_1}{|x-q|}$ is non-positive, so by \eqref{eq:vare2} and monotone convergence we obtain
$$
\n H_s^\Om(q)\cdot e_1=-\frac{2(n+2s)}{\om_{n-2}}\int_{\Om\setminus \Om'} \frac{(x-q)\cdot e_1}{|x-q|^{n+2s+2}}\,dx\,.
$$
Since $|\n H_s^\Om(q)\cdot e_1| \leq \de_s(\Om)$ and
$$
-\frac{(x-q)\cdot e_1}{|x-q|^{n+2s+2}} \geq \frac{|x_1-q_1|}{\diam(\Om)^{n+2s+2}}= \frac{|x_1-\l|}{\diam(\Om)^{n+2s+2}}\qquad \mbox{on $\Om\setminus\Om'\subset \{x_1\leq 0\}$}\,,
$$
we finally get
$$
\de_s(\Om) \geq \frac{2(n+2s)}{\diam(\Om)^{n+2s+2}\om_{n-2} }
\int_{\Om' \setminus \Om} |x_1-\l| \,dx=\frac{(n+2s)}{\diam(\Om)^{n+2s+2}\om_{n-2} }
\int_{\Om' \Delta \Om} |x_1-\l| \,dx\,,
$$
which completes the proof of \eqref{bound_deficit_one_direct}.\\

We now prove (a). For this it is enough to combine \eqref{bound_deficit_one_direct} with Chebyshev's inequality to get
$$
\big|\big\{x\in \Om \triangle \Om':\ \dist(x,\pi_\l)\geq \gamma \big\} \big| \leq \frac{1}{\gamma}\,\frac{\om_{n-2} }{n+2s}\, \diam(\Om)^{n+2s+2} {\de_s(\Om)}\,,
$$
that together with the trivial bound
$$
\big| \big\{x\in \Om \triangle \Om':\ \dist(x,\pi_\l) \leq \gamma\big\} \big| \leq  2\,\gamma\,\diam (\Om)^{n-1}\,,
$$
gives us \eqref{diff_symm_small} choosing $\gamma=\sqrt{\frac{\om_{n-2} }{2(n+2s)}}\,\diam(\Om)^{s+(3/2)}\sqrt{\de_s(\Om)}$.\\

If we know in addition that $\dist(\zero,\pi_\l)\leq 1/8$ and that $B_r(\zero)\subset \Om\subset B_R(\zero)$
for some radii satisfying \eqref{eq:radii}, then
we can use the stronger bound
$$
\big| \big\{x\in \Om \triangle \Om':\ \dist(x,\pi_\l) \leq \gamma\big\} \big| \leq  C(n)\,\gamma\,(R-r)\qquad \forall\, \gamma\leq 1/4,
$$
so \eqref{C1} follows by choosing $\gamma=\sqrt{\frac{\delta_s(\Om)}{R-r}}$.
\end{proof}

We now deduce Theorem \ref{thm rigidity} from Proposition \ref{prop:1direction}.

\begin{proof}[Proof of Theorem \ref{thm rigidity}]
We begin by noticing that, thanks to the regularity theory developed in \cite{BFV}
(see in particular the proof of \cite[Theorem 1]{BFV}),
$C^{1,2s}$
 domains with constant nonlocal mean curvature are actually $C^\infty$, so
 Proposition \ref{prop:1direction} applies.
In particular, since by assumption $\de_s(\Om)=0$, Proposition \ref{prop:1direction} implies that $\Om$ is symmetric in any direction.

Since the barycenter $\bb$ of $\Om$ belongs to every axis of symmetry and every rotation can be written as a composition of reflections, we have that $\Om$ is invariant under rotations, which implies that $\pa\Om$ is a collection of concentric spheres centered at $\bb$. To show that $\pa\Om$ is just one sphere, we apply again the method of moving planes in an arbitrary direction: if $\pa\Om$ is not connected then the critical hyperplane must be a hyperplane of symmetry and cannot contain $\bb$, which is a contradiction. Hence $\pa\Om$ must have a single connected component, i.e., $\pa\Om$ is a sphere.
\end{proof}

\section{Stability}\label{sec:stab} 
Before proving Theorems \ref{thm:stability} and \ref{thm:stability2} we first show the following lemma
stating that if $\de_s(\Om)$ is small then, up to a translation, all critical planes from the moving planes method pass close to the origin.
Again, as in Proposition \ref{prop:1direction}, it will be important to show a stronger bound when $\Om$ is comparable to a ball of radius $1$.

\begin{lemma}\label{lemma direzioni} Let $\Om$ be an open bounded set of class $C^{2,\a}$ for some $\a>2s$ with
  \begin{equation} \label{delta_small}
\frac{\diam(\Om)^{n+s+(1/2)}}{|\Om|}\,\sqrt{\de_s(\Om)} \leq \min\Bigl\{\frac14,\frac 1n\Bigr\}\,\sqrt{\,\frac{n+2s}{8\,\om_{n-2} }}\,,
\end{equation}
and suppose that the critical planes with respect to the coordinate directions $\pi_{e_i}$ coincide with $\{x_i=0\}$ for every $i=1,...,n$.
Also, given $e\in\SS^{n-1}$, denote by $\l_e$ the critical value  associated to $e$ as in \eqref{def:lam_critico}.
\begin{enumerate}
\item[(a)] The bound 
\begin{equation} \label{dist_zero_pi_small}
|\l_e|\leq C_2 \sqrt{\de_s(\Om)} 
\end{equation}
holds with
 $$C_2=4\,(n+3)\,\frac{\diam(\Om)^{n+s+(3/2)}}{|\Om|}\,C_1,$$ where $C_1$ is as in \eqref{def:C0}.
\item[(b)]
Assume in addition that $\dist(\zero,\pi_\l)\leq 1/8$
and $B_r(\zero)\subset \Om\subset B_R(\zero)$ for some radii satisfying
\eqref{eq:radii}. Then
\begin{equation} \label{dist_zero_pi_small 2}
|\l_e|\leq C^*(n) \sqrt{\de_s(\Om)}\,\sqrt{R-r} \,
\end{equation}
for some dimensional constant $C^*(n)$.
\end{enumerate}
\end{lemma}

\begin{proof}
We first prove (a).
To this aim, we define $\Om^\zero=\{-x:x\in\Om\}$ and set
\begin{equation}
\label{eq:C1*}
C_1^*= C_1\,\diam(\Om)^{n+s+(1/2)},
\end{equation}
where $C_1$ is defined as in \eqref{def:C0}.
Then, since $\Om^\zero$ can be obtained from $\Om$ by symmetrizing it with respect to the hyperplanes $\{x_i=0\}=\pi_{e_i}$ for $i=1,\ldots,n$,
 applying Proposition \ref{prop:1direction} with respect to the coordinate directions we obtain
\begin{equation} \label{diff_symm_zero_small}
|\Om \triangle \Om^\zero| \leq  n\, C_1^*\, \sqrt{\de_s(\Om)}\,.
\end{equation}
Now, to prove \eqref{dist_zero_pi_small} we assume 
that $\l_e >0$ (the case $\l_e<0$ being similar). We first note that
\begin{equation} \label{Lambda_e_bound}
\L_e=\sup\{x\cdot e: x\in \Om\}\leq \diam(\Om)\,.
\end{equation}
Indeed, if $\L_e>\diam(\Om)$, then $x\cdot e\ge0$ for every $x\in\Om$, and thus $|\Om \triangle \Om^\zero | = 2 |\Om|$, which  contradicts \eqref{diff_symm_zero_small} and \eqref{delta_small}. This said, we denote by $\Om'$ the reflection of $\Om$ about the critical hyperplane $\pi_{\l_e}$, and deduce from Proposition \ref{prop:1direction} that
\begin{equation}
  \label{approx symmetry}
  |\Om\Delta\Om'|\le C_1^* \sqrt{\de_s(\Om)}\,.
\end{equation}
Now, recalling the notation $\Omega_\mu=\Omega\cap \E_\mu= \Omega\cap \{x\cdot e >\mu\}$,
it follows by \eqref{approx symmetry} (which tells us that $\Om$ is almost symmetric with respect to $\pi_{\l_e}$) that
\begin{equation}
\label{eq:E1}
|\Omega_{\l_e}| \geq \frac{|\Om|}2 -C_1^*\,\sqrt{\de_s(\Om)} \,.
\end{equation}
Since $\Om$ is almost symmetric about $\zero$ by \eqref{diff_symm_zero_small}, 
using the notation $\E_{\l_e}^\zero=\{-x:x\in\E_{\l_e}\}$ we see that \eqref{eq:E1} gives
$$
|\Omega \cap \E_{\l_e}^\zero|=|\Omega^\zero \cap \E_{\l_e}| \geq
|\Omega_{\l_e}| -|\Omega \Delta \Omega^\zero|\geq \frac{|\Om|}2 -(n+1)C_1^*\sqrt{\de_s(\Om)}\,,
$$
which together with \eqref{eq:E1} implies
\begin{equation}
\label{eq:no mass strip}
|\{x \in \Om:\ -\l_e \leq x \cdot e \leq \l_e \}|  \leq (n+2)C_1^* \sqrt{\de_s(\Om)} \,.
\end{equation}
In other words, by combining the almost-symmetry of $\Om$ with respect to $\zero$ and to $\pi_{\l_e}$ we have shown that $\Om$ has small volume in the strip $\{|x\cdot e|\le\l_e\}$. Since $\{\l_e\le x\cdot e\le 3\l_e\}$ is mapped into $\{|x\cdot e|\le\l_e\}$ by the reflection with respect to $\pi_{\l_e}$, exploiting again \eqref{approx symmetry} and \eqref{eq:no mass strip} we get
\begin{equation}
\label{eq:no mass}
\begin{split}
|\{x \in \Om:\ \l_e < x \cdot e  < 3 \l_e \}|& =|\{x \in \Om':\ |x\cdot e|\leq \l_e \}|\\
&\leq |\{x \in \Om:\ |x\cdot e|\leq \l_e \}|+|\Om\Delta\Om'|\leq (n+3)C_1^* \sqrt{\de_s(\Om)}.
\end{split}
\end{equation}
Define now
$$
m_k:=|\{x \in \Om:\ (2k-1)\l_e \leq x \cdot e \leq (2k+1)\l_e \}|\,,\qquad k\ge 1\,,
$$
and notice that, by the moving planes procedure, the set $\Omega\cap\pi_\mu$ (seen as a subset of $\R^{n-1}$) is included inside
$\Omega\cap\pi_{\mu'}$ whenever $\l_e \leq \mu' \leq \mu$. In particular
the function $\mu \mapsto \mathcal H^{n-1}(\Omega\cap\pi_\mu)$ is decreasing on $(\l_e,\L_e)$, hence $m_k$ is a decreasing sequence and  \eqref{eq:no mass} gives us
$$
m_k \leq m_1 \leq (n+3)C_1^* \sqrt{\de_s(\Om)}\qquad \forall\, k \geq 1.
$$
Recalling that $\Omega \subset \{x\cdot e \leq \L_e\}$, combining this last estimate with \eqref{eq:no mass strip} and letting $k_0$ be the smallest natural number such that $(2k_0+1)\l_e\ge\L_e$ we get
$$
|\Omega_{\l_e}|=
|\Omega\cap \{\l_e \leq x\cdot e \leq \L_e\}|\leq
\sum_{k=1}^{k_0}\,m_k \leq \frac12\Big(\frac{\L_e}{\l_e}+1\Big)\,(n+3)C_1^* \sqrt{\de_s(\Om)}\,,
$$
hence (thanks to \eqref{Lambda_e_bound})
\[
|\Om_{\l_e}|\,\l_e\le (n+3)\,C_1^*\,\diam(\Om)\,\sqrt{\de_s(\Om)}\,.
\]
Since $|\Om_{\l_e}|\ge|\Om|/4$ (by \eqref{eq:E1} and \eqref{delta_small}), recalling \eqref{eq:C1*} we get \eqref{dist_zero_pi_small}.
\\

To prove (b) it suffices to observe that,
under the assumption that $\dist(\zero,\pi_\l)\leq 1/8$
and $B_r(\zero)\subset \Om\subset B_R(\zero)$ with $r,R$ satisfying \eqref{eq:radii}, we can repeat the very same proof done above but using \eqref{C1}
in place of \eqref{diff_symm_small} to obtain \eqref{dist_zero_pi_small 2}.
\end{proof}

We now prove Theorems \ref{thm:stability} and \ref{thm:stability2}.

\begin{proof}[Proof of Theorem \ref{thm:stability}] {\it Step 1: proof of \eqref{stima stab}.}
Up to a translation, we can assume that the critical planes with respect to the coordinate directions $\pi_{e_i}$ coincide with $\{x_i=0\}$ for every $i=1,...,n$.

Notice that, since $\rho(\Om)\le 1$ and $\sqrt{\eta_s(\Om)}=\frac{\diam(\Om)^{n+s+(1/2)}}{|\Om|}\,\sqrt{\de_s(\Om)}$, one can directly assume that \eqref{delta_small} holds. Moreover, setting
\begin{equation} \label{def:ri_re}
r=\min_{x\in \pa\Om} |x|\,,\qquad R=\max_{x\in \pa\Om} |x|\,,
\end{equation}
it is enough to control $R-r$ (as it gives an upper bound on $\rho(\Om)$).

Let $x,y \in \pa\Om$ be such that $|x|=r$ and $|y|=R$. Assuming without loss of generality that $x \neq y$, we consider the unit vector
\[
e=\frac{y-x}{|y-x|}\,,
\]
and let $\pi_{\l_e}$ denote the corresponding critical hyperplane. We notice that $y$ is closer than $x$ to the critical hyperplane $\pi_{\l_e}$, i.e.,
\begin{equation} \label{dist_x_dist_y}
\dist(x,\pi_{\l_e}) \geq \dist(y,\pi_{\l_e})\,.
\end{equation}
Indeed, since $x=y-te$ with $t=|x-y|$,  the method of moving planes implies that the critical position can be reached at most when $y'$ (the reflection of $y$ with respect to $\pi_{\l_e}$) is tangent to $x$, which corresponds to the equality case in \eqref{dist_x_dist_y},
while in all the other cases strict inequality holds. Thus, by \eqref{dist_x_dist_y}
and the fact that $e$ is parallel to $y-x$ we get
\begin{equation}
\label{eq:Rr lambda}
R-r=|y|-|x| \leq 2 \dist(\zero,\pi_{\l_e})=2|\l_e|\,
\end{equation}
that combined with \eqref{dist_zero_pi_small} implies
that 
\begin{equation}
  \label{viavia}
R-r\leq 2\,C_2^* \sqrt{\de_s(\Om)}=16\,(n+3)\,\sqrt{\frac{2\,\om_{n-2} }{n+2s}}\,\eta_s(\Om).
\end{equation}
Now, since all the quantities involved are scaling invariant, we rescale $\Om$ so that $R=1$,
and we assume without loss of generality that
$$
R-r \geq 16\,\delta_s(\Om)
$$
as otherwise \eqref{stima stab} trivially holds.
In this way it follows from \eqref{viavia} that
 \eqref{eq:radii} holds provided  $\eta_s(\Om)$
is small enough.
Also, thanks to \eqref{dist_zero_pi_small} we see that
$\dist(\zero,\pi_{\l_e})\leq 1/8$ for all $e \in \SS^{n-1}$
if $\delta_s(\Om)$ (or equivalently $\eta_s(\Om)$) is sufficiently small.

Hence, this allows us to combine \eqref{eq:Rr lambda} with \eqref{dist_zero_pi_small 2} to get
$$
  R-r\leq 2\,C^*(n) \sqrt{\de_s(\Om)}\,\sqrt{R-r},
$$
that is
\begin{equation}
\label{eq:Rr delta}
R-r \leq  4\,C^*(n)\,\delta_s(\Om),
\end{equation}
which proves \eqref{stima stab}.
\medskip

\noindent {\it Step 2: a quantitative Lipschitz bound on $\pa \Om$}. 
We want to show that if $\eta_s(\Om)\le\eta(n)$ for some dimensional constant $\eta(n)$,
then $\pa \Om$ is Lipschitz-flat with a uniform bound.

Since all the quantities involved are scaling invariant, we assume as at the end of step 1 that $R=1$ so that
$$
B_{r}(\zero)\subset \Om\subset B_{1}(\zero)
$$
with 
\begin{equation}
\label{eq:1r}
1-r \leq C(n)\,\eta_s(\Om)
\end{equation}
(by \eqref{eq:Rr delta}),
and then prove \eqref{lipschitz map} for $\eta_s(\Om)$ small enough.

To this end, it is enough to
show that there exists a dimensional constant $M=M(n)$ such that, for any $x \in \pa \Om$ and $y \in \pa B_{1 -M\eta_s(\Om)}(\zero)$ such that  the ``open'' segment $(x,y)$
is contained outside $B_{1 -M\eta_s(\Om)}(\zero)$, then
$(x,y)\subset {\Om}$. Indeed, this means that for any $x \in \pa \Om$ we can find a uniform cone
of opening $\pi - C\sqrt{\eta_s(\Om)}$ with tip at $x$ and axis parallel to $\frac{x}{|x|}$
which is contained inside $\Om$, and this implies that $\pa \Om$ is locally the graph of a Lipschitz function
satisfying \eqref{lipschitz map}.

Now, to prove the latter fact, assume by contradiction that there exist $x\in\pa\Om$ and $y\in \pa B_{1-M\eta_s(\Om)}(\zero)$ for which there exists
a point $z \in (x,y)\cap \Om^c$.
Set $e=\frac{x-y}{|x-y|}$ and notice that, since $z \in \Om^c$, it follows that the moving planes method has to stop before reaching $z$, that is 
$\l_e\ge z\cdot e$.
Now, since $(x,y)\subset B_{1}(\zero)\setminus B_{1-M\eta_s(\Om)}(\zero)$ and $y \in \pa B_{1-M\eta_s(\Om)}(\zero)$,
we have $y \cdot e \geq 0$.
Hence, since $z-y$ is parallel to $e$ and $z \in \Om^c \subset B_{r}(\zero)^c$ we get
$$
\l_e \geq (z-y)\cdot e+y\cdot e \geq (z-y)\cdot e=|z-y| \geq M\,\eta_s(\Om) -(1-r).
$$
On the other hand \eqref{dist_zero_pi_small} gives
\[
C(n)\,\eta_s(\Om)\ge |\l_e|
\]
(recall that $1\leq \diam(\Om)\le 2$ and $s \in (0,1)$), which
leads to a contradiction to \eqref{eq:1r} provided $M=M(n)$ is large enough.
\end{proof}

\begin{proof}[Proof of Theorem \ref{thm:stability2}]
Our goal here it to exploit the results from Theorem \ref{thm:stability} to get closeness to a ball in higher norms.
For this, we need to show that our assumptions on $H_s^\Om$
imply that $\pa\Om$ is smooth with some {\em quantitative bounds} depending only on $\eta_s(\Om)$.
Hence, we first formulate the following regularity criterion that is implicitly contained in \cite{CG}
(recall that definition of $\C_r$ and $\D_r$ from \eqref{eq:DC}).

Given $n\ge2$, $s,\ell\in(0,1/2)$, and $\b\in(0,2s)$, there exist positive constants $\e=\e(n,s,\ell,\beta)$ and $C_*=C_*(n,s,\ell,\b)$ with the following property: Let $E$ be an open set with $C^{2}$-boundary such that for some $L\ge0$ it holds
\begin{eqnarray}\label{ipotesi su E}
  \|H^E_s\|_{C^0(\pa E)}\le L\,,\qquad
  \frac{|B_{r}(y)\cap E|}{\om_n\,r^n}\in(\ell,1-\ell)\,,\qquad\forall \,y\in \pa E\,,r<\ell\,.
\end{eqnarray}
If $\zero\in\pa E$ and $ r<\ell$ are such that
\begin{equation}\label{flatness hp}
  B_r(\zero)\cap\pa E\subset\big\{x\in\R^n:|x_n|\le \e r\big\}\,,\qquad L\, r\le \e\,,
\end{equation}
then there exists $u\in C^{1,\b}(\D_{ r/2})$ such that
\[
\C_{ r/2}\cap\pa E=(\Id\times u)(\D_{ r/2})\,,
\]
with
\[
\|\nabla u\|_{C^0(\D_{ r/2})}+ r^\b\,[\nabla u]_{C^{0,\b}(\D_{ r/2})}\le C_*\,\Bigl(\frac{\|u\|_{C^0(\D_{ r})}}r+L\,r\Bigr)\,.
\]


\medskip

\noindent {\it Step 1: uniform $C^{2,\gamma}$ bounds on $\pa \Om$}. We show that the regularity criterion stated above applies with $E=\Om$. 

Since, by the definition of $\rho(\Om)$, the radii $1-2\rho(\Om)$ and $1$ must be optimal for the inclusion \eqref{eq:scale omega} to hold,
we can find points $p_1 \in \pa\Om\cap \pa B_{1-2\rho(\Om)}(\zero)$ and $p_2 \in \Om\cap \pa B_{1}(\zero)$. Hence,
it follows by the inclusions \eqref{eq:scale omega} and \eqref{def:Hs} that
$$
H_s^\Om(p_1) \leq H_s^{B_{1-2\rho(\Om)}},\qquad H_s^\Om(p_2)\geq H_s^{B_1},
$$
and because the Lipschitz constant of $H_s^\Om$ is bounded by $\delta_s(\Om)\leq C\, \eta_s(\Om)$
and
$$
|H_s^{B_{1-2\rho(\Om)}}-H_s^{B_1}|\leq \,C\,\rho(\Om) \leq C\, \eta_s(\Om)
$$
(by \eqref{stima stab}),
 we deduce that 
\begin{equation}
\label{eq:mean constant}
\bigl\|H_s^\Om-H_s^{B_1}\bigr\|_{L^\infty(\pa\Om)}  \leq C\,\eta_s(\Om).
\end{equation}

Notice now that the uniform Lipschitz estimate provided by Theorem \ref{thm:stability} implies that the density estimates in 
\eqref{ipotesi su E} hold.
Thus, provided $\eta_s(\Om)$ is small enough, \eqref{ipotesi su E} holds with $L=2\,H_s^{B_1}$ and for some $\ell=\ell(n)>0.$
At the same time we can find $r=r(n)>0$, depending on $\pa B_1(\zero)$ only, such that if $x\in \pa B_1(\zero)$ then
\begin{equation}\label{flatness hp palla}
  B_{2 r}(x)\cap\pa B_1(\zero)\subset\Big\{y\in\R^n:\Big|(y-x)\cdot\frac{x}{|x|}\Big|\le \frac{\e r}2\Big\}\,,\qquad L\, r\le {\e}\,.
\end{equation}
Hence, assuming that $\eta_s(\Om)$ is small enough in terms of $r$, by \eqref{stima stab} and \eqref{eq:scale omega} we can ensure that
\begin{equation}\label{flatness hp Omega}
  B_{r}(z)\cap\pa \Om\subset\Big\{y\in\R^n:\Big|(y-z)\cdot\frac{z}{|z|}\Big|\le \e r \Big\}\qquad \forall\,z \in \pa \Om\,,
\end{equation}
and applying the regularity criterion stated before we obtain that, for any $z \in \pa \Om$, there exists a uniform neighborhood such that,
in a suitable system of coordinates,
$\pa \Om$ is given by the graph of a function $u_z:\D_{r} \to \R$ with
$$
\|u_z\|_{C^{1,\beta}(\D_{r/2})} \leq C(n,s,\beta).
$$
Now, 
choosing $\beta$ arbitrarily close to $2s$ and exploiting the fact that $H^\Om_s\in C^{0,\g}(\pa\Om)$ for every $\g\in(0,1)$ together with the higher regularity theory by \cite[Section 3]{BFV}, we obtain that 
$$
\|u_z\|_{C^{2,\tau}(\D_{r/4})} \leq C(n,s,\tau)
$$
for any $\tau <2s$.

\medskip
\noindent {\it Step 2: $\pa \Om$ is $C^2$-close to a sphere linearly in $\eta_s(\Om)$}.
By the previous step we know that there exists a map
$f:\pa B_1(\zero)\to \R$ of class $C^{2,\tau}$ for any $\tau<2s$ satisfying
$$
\|f\|_{C^{2,\tau}(\pa B_1(\zero))}\le C(n,s,\tau)
$$
and such that $\pa\Om=\{y+f(y)\,y\,:\,y \in \pa B_1(\zero)\bigr\}$.
Notice that, by \eqref{eq:scale omega},
\begin{equation}
\label{eq:infty f}
\|f\|_{L^\infty(\pa B_1(\zero))}\le C(n)\,{\eta_s(\Om)}\,,
\end{equation}
so we deduce by interpolation that for any $\zeta<2s$ there exists an exponent $\a(\zeta)>0$ such that
\begin{equation}
\label{eq:C2zeta}
\|f\|_{C^{2,\zeta}(\pa B_1(\zero))}\le C(n,s,\zeta)\,\eta_s(\Om)^{\a(\zeta)}\,.
\end{equation}
This implies in particular that $\pa \Om$ is $C^2$-close to a sphere, so $\Om$ is convex for $\eta_s(\Om)$
sufficiently small.
We now want to show that \eqref{eq:C2zeta} is still valid if we replace $\a(\zeta)$ with $1$,
which will prove the theorem with $F(y)=y+f(y)\,y$.

For this, we write the nonlocal mean curvature in terms of $f$ starting from \eqref{def:Hsdiv}:
in this way, since any point $x \in \pa \Om$ can be written as $y+f(y)\,y$ with $y \in \pa B_1(\zero)$,
by the area formula we get that, at the point $p=q+f(q)\,q \in \pa \Om$,
$$
H_s^\Om(p)=\frac{1}{s\,\om_{n-2}}\int_{\pa B_1(\zero)} \frac{y+f(y)\,y - q-f(q)\,q}{|y+f(y)\,y - q-f(q)\,q|^{n+2s}} \cdot \Bigl(y-\frac{\n_T f(y)}{1+f(y)}\Bigr)\,\bigl(1+f(y)\bigr)^{n-1}\,d\H^{n-1}_y.
$$
To simplify the notation we define the vector-field  $v_q(y):=y+f(y)\,y - q-f(q)\,q$, so that the above expression becomes
\begin{align*}
H_s^\Om(p)&=\frac{1}{s\,\om_{n-2}}\int_{\pa B_1(\zero)} \frac{v_q(y)}{|v_q(y)|^{n+2s}}\\
&\qquad\qquad\qquad\qquad \cdot \Bigl(y
\bigl(1+f(y)\bigr)^{n-1}
-\frac{1}{n-1}\n_T\Bigl[ \bigl(1+f(y)\bigr)^{n-1}- \bigl(1+f(q)\bigr)^{n-1}\Bigr]\Bigr)\,d\H^{n-1}_y.
\end{align*}
Now, noticing that the normal to $\pa B_1(\zero)$ at $y$ is equal to $y$ itself, by the tangential divergence theorem
(see for instance \cite[Theorem 11.8]{Maggi}) 
we get (notice that the classical mean curvature of $\pa B_1(\zero)$ is $n-1$)
\begin{align*}
H_s^\Om(p)&=\frac{1}{s\,(n-1)\,\om_{n-2}}\int_{\pa B_1(\zero)} \Div_T\biggl(\frac{v_q(y)}{|v_q(y)|^{n+2s}} \biggr)\Bigl[ \bigl(1+f(y)\bigr)^{n-1}- \bigl(1+f(q)\bigr)^{n-1}\Bigr]\,d\H^{n-1}_y\\
&+\frac{1}{s\,\om_{n-2}}\int_{\pa B_1(\zero)} \frac{v_q(y)\cdot y}{|v_q(y)|^{n+2s}}\,\bigl(1+f(q)\bigr)^{n-1}\,d\H^{n-1}_y.
\end{align*}
Since $\Div_T(y)=n-1$ and $\n_Tf(y)\cdot y=0$ we have
$$
\Div_T\biggl(\frac{v_q(y)}{|v_q(y)|^{n+2s}} \biggr)=
\frac{(n-1)\bigl(1+f(y)\bigr)}{|v_q(y)|^{n+2s}}
-(n+2s)\frac{v_q(y)\cdot \n_T|v_q(y)| }{|v_q(y)|^{n+2s+1}}.
$$
So, computing
$$
\n_T v_q(y)= \bigl(1+f(y)\bigr)\,\n_T y+\n_Tf(y)\otimes y
$$
and denoting by $\pi_y:\R^n\to \R^n$ the orthogonal projection onto $y^\perp$, we get
\begin{align*}
v_q(y)\cdot \n_T |v_q(y)|&=\frac{v_q(y)\cdot \n_T v_q(y)\cdot v_q(y)}{|v_q(y)|}\\
&=\frac{\bigl(1+f(y)\bigr)\,|\pi_yv_q(y)|^2-\bigl(1+f(q)\bigr)\,\bigl((q-y)\cdot \n_Tf(y)\bigr)\,\bigl(v_q(y)\cdot y\bigr)}{|v_q(y)|}.
\end{align*}
Thanks to the elementary identity 
\begin{equation}
\label{eq:pitagora}
(y-q)\cdot y=1-q\cdot y=\frac12 |y-q|^2
\end{equation}
we see that
$$
|\pi_yv_q(y)|^2=\bigl(1+f(q)\bigr)^2|\pi_yq|^2=\bigl(1+f(q)\bigr)^2\bigl(1+y\cdot q\bigr)\,\frac{|y-q|^2}{2},
$$
and
$$
v_q(y)\cdot y=f(y)-f(q)+\bigl(1+f(q)\bigr)\,\frac{|y-q|^2}{2}.
$$
Hence, 
setting for simplicity
$$
\Gamma_f(y,q)=\bigl(1+f(y)\bigr)^{n-1}- \bigl(1+f(q)\bigr)^{n-1} \,,
$$
and combining all these formulas, we obtain
\begin{align*}
H_s^\Om(p)&=\frac{1}{s\,\om_{n-2}}\int_{\pa B_1(\zero)} \frac{1+f(y)}{|v_q(y)|^{n+2s}} \,\Gamma_f(y,q)\,d\H^{n-1}_y\\
&-\frac{n+2s}{2s\,(n-1)\,\om_{n-2}}\int_{\pa B_1(\zero)} \frac{\bigl(1+f(y)\bigr)\,(1+y\cdot q)\,\bigl(1+f(q)\bigr)^2\,|y-q|^2}{|v_q(y)|^{n+2s+2}} \,\Gamma_f(y,q)\,d\H^{n-1}_y\\
&+\frac{n+2s}{s\,(n-1)\,\om_{n-2}}\int_{\pa B_1(\zero)} \frac{\bigl(1+f(q)\bigr)\,\bigl((q-y)\cdot \n_Tf(y)\bigr) \,\bigl(f(y)-f(q)\bigr)}{|v_q(y)|^{n+2s+2}} \,\Gamma_f(y,q)\,d\H^{n-1}_y\\
&+\frac{n+2s}{2s\,(n-1)\,\om_{n-2}}\int_{\pa B_1(\zero)} \frac{\bigl(1+f(q)\bigr)^2\,\bigl((q-y)\cdot \n_Tf(y)\bigr)\,|y-q|^2}{|v_q(y)|^{n+2s+2}} \,\Gamma_f(y,q)\,d\H^{n-1}_y\\
&+\frac{1}{s\,\om_{n-2}}\,\bigl(1+f(q)\bigr)^{n-1}\int_{\pa B_1(\zero)} \frac{f(y)-f(q)}{|v_q(y)|^{n+2s}}\,d\H^{n-1}_y\\
&+\frac{1}{2s\,\om_{n-2}}\,\bigl(1+f(q)\bigr)^{n}\int_{\pa B_1(\zero)} \frac{|y-q|^2}{|v_q(y)|^{n+2s}}\,d\H^{n-1}_y.
\end{align*}
Noticing that
$$
1+y\cdot q=2-\frac{1}{2}|y-q|^2,
$$
the above expression can be rewritten as
\begin{equation}
\label{eq:mean f}
\begin{split}
H_s^\Om(p)&=\frac{1}{s\,\om_{n-2}}\int_{\pa B_1(\zero)} \frac{1+f(y)}{|v_q(y)|^{n+2s}} \,\Gamma_f(y,q)\,d\H^{n-1}_y\\
&-\frac{n+2s}{s\,(n-1)\,\om_{n-2}}\int_{\pa B_1(\zero)} \frac{\bigl(1+f(y)\bigr)\,\bigl(1+f(q)\bigr)^2\,|y-q|^2}{|v_q(y)|^{n+2s+2}} \,\Gamma_f(y,q)\,d\H^{n-1}_y\\
&+\frac{n+2s}{s\,(n-1)\,\om_{n-2}}\int_{\pa B_1(\zero)} \frac{\bigl(1+f(q)\bigr)\,\bigl((q-y)\cdot \n_Tf(y)\bigr) \,\bigl(f(y)-f(q)\bigr)}{|v_q(y)|^{n+2s+2}} \,\Gamma_f(y,q)\,d\H^{n-1}_y\\
&+\frac{n+2s}{2s\,(n-1)\,\om_{n-2}}\int_{\pa B_1(\zero)} \frac{\bigl(1+f(q)\bigr)^2\,\bigl((q-y)\cdot \n_Tf(y)\bigr)\,|y-q|^2}{|v_q(y)|^{n+2s+2}} \,\Gamma_f(y,q)\,d\H^{n-1}_y\\
&+\frac{1}{s\,\om_{n-2}}\,\bigl(1+f(q)\bigr)^{n-1}\int_{\pa B_1(\zero)} \frac{f(y)-f(q)}{|v_q(y)|^{n+2s}}\,d\H^{n-1}_y\\
&+\frac{1}{2s\,\om_{n-2}}\,\bigl(1+f(q)\bigr)^{n}\int_{\pa B_1(\zero)} \frac{|y-q|^2}{|v_q(y)|^{n+2s}}\,d\H^{n-1}_y\\
&+\frac{1}{4s\,\om_{n-2}}\,\int_{\pa B_1(\zero)} \frac{\bigl(1+f(y)\bigr)\,\bigl(1+f(q)\bigr)^2\,|y-q|^4}{|v_q(y)|^{n+2s+2}} \,\Gamma_f(y,q)\,d\H^{n-1}_y.
\end{split}
\end{equation}
We now notice that, since
$$
\Gamma_f(y,q)=(n-1)[f(y)-f(q)] \Bigl(1+ P\bigl(f(y),f(q)\bigr)\Bigr)
$$
with $P(t,s)$ a polynomial of degree $n-2$ which vanishes at $t=s=0$,
the first five terms in the right hand side above can be written as
$$
-\int_{\pa B_1(\zero)} \bigl(f(y)-f(q) \bigr)\, K(y,q)\,d\H^{n-1}_y
$$
where the kernel $K(y,q)$ behaves like a $C^{1,\tau}$ perturbation of the $\frac{1+2s}{2}$-fractional Laplacian on $\R^{n-1}$: more precisely
\begin{equation}
\label{eq:K}
K(y,q)=\frac{2}{\om_{n-2}} \frac{1}{|y-q|^{(n-1)+(1+2s)}}\Bigl(1+G_f(y,q)\Bigr),
\end{equation}
where $G_f:\pa B_1(\zero)\times \pa B_1(\zero)\to \R$ is a $C^{1,\tau}$-function (depending on $f$) which satisfies
$$
\|G_f\|_{C^{1,\tau}(\pa B_1(\zero)\times \pa B_1(\zero))} \leq C\,\|f\|_{C^{2,\tau}(\pa B_1(\zero))}\quad \forall\,\tau\in[0,2s).
$$

We now subtract the value of the above expression in the right hand side of \eqref{eq:mean f} at $f=0$ (which corresponds to the case of the unit sphere) to get
\begin{equation}
\label{eq:Kf}
H_s^\Om\bigl(F(q)\bigr)-H_s^{B_1}=-\int_{\pa B_1(\zero)} \bigl(f(y)-f(q) \bigr)\, K(y,q)\,d\H^{n-1}_y+g(q)\,,
\end{equation}
where $F(q)=q+f(q)\,q$ and
\begin{align*}
g(q)&=\frac{1}{2s\,\om_{n-2}}\,\biggl(\bigl(1+f(q)\bigr)^{n}\int_{\pa B_1(\zero)} \frac{|y-q|^2}{|v_q(y)|^{n+2s}}\,d\H^{n-1}_y
-\int_{\pa B_1(\zero)} \frac{1}{|y-q|^{n+2s-2}}\,d\H^{n-1}_y\Biggr)\\
&+\frac{1}{4s\,\om_{n-2}}\,\int_{\pa B_1(\zero)} \frac{\bigl(1+f(y)\bigr)\,\bigl(1+f(q)\bigr)^2\,|y-q|^4}{|v_q(y)|^{n+2s+2}} \,\Gamma_f(y,q)\,d\H^{n-1}_y
\end{align*}
is a $C^1$ function satisfying
$$
\|g\|_{L^{\infty}(\pa B_1(\zero))} \leq C\,\|f\|_{L^{\infty}(\pa B_1(\zero))},\qquad \|g\|_{C^1(\pa B_1(\zero))} \leq C\,\|f\|_{C^1(\pa B_1(\zero))}.
$$
Since $K$ is a $C^1$ perturbation of the $\frac{1+2s}{2}$-fractional Laplacian,
applying \cite[Theorem 61]{CSi} locally in charts (using a cut-off function) we deduce that
$$
\|f\|_{C^{1,\tau}(\pa B_1(\zero))} \leq C(n,s,\tau)\Bigl(\|f\|_{L^{\infty}(\pa B_1(\zero))}+ \|g\|_{L^{\infty}(\pa B_1(\zero))} + \|H_s^\Om\circ F-H_s^{B_1}\|_{L^{\infty}(\pa B_1(\zero))}\Bigr)\qquad \forall\,\tau <2s.
$$
Also, differentiating \eqref{eq:Kf} we can apply the same result to the first derivatives of $f$ 
(see for instance \cite[Section 2.4]{BFV} for more details on how this differentiation argument works)
to get
$$
\|f\|_{C^{2,\tau}(\pa B_1(\zero))} \leq C(n,s,\tau)\Bigl(\|f\|_{C^1(\pa B_1(\zero))}+ \|g\|_{C^1(\pa B_1(\zero))}+ \|H_s^\Om\circ F-H_s^{B_1}\|_{C^1(\pa B_1(\zero))}\Bigr)\qquad \forall\,\tau <2s.
$$
Notice now that by \eqref{eq:mean constant}, the definition of $\delta_s(\Om)$, and the fact that $\|F\|_{C^1(\pa B_1(\zero))} \leq C$, we have
$$
 \|H_s^\Om\circ F-H_s^{B_1}\|_{C^1(\pa B_1(\zero))} \leq C\,\delta_s(\Om)  \,.
$$
Hence combining all these estimates and recalling  \eqref{eq:infty f}, we  conclude that
$$
\|f\|_{C^{2,\tau}(\pa B_1(\zero))} \leq C(n,s,\tau)\Bigl(\delta_s(\Om)+\|f\|_{C^0(\pa B_1(\zero))}\Bigr)
\leq C(n,s,\tau)\,{\eta_s(\Om)}
\qquad \forall\,\tau <2s,
$$
as desired.

\end{proof}

%
%
%

\end{document}